\documentclass[12pt]{amsart}
\usepackage{amssymb}
\usepackage{graphicx}
\newtheorem{thm}{Theorem}
\newtheorem{prop}{Proposition}
\newtheorem{lemma}{Lemma}
\newtheorem{cor}{Corollary}
\newcommand{\im}{\operatorname{im}}

\newcommand{\End}{\operatorname{End}}
\newcommand{\coker}{\operatorname{coker}}
\newcommand{\T}{\mathcal{T}}
\DeclareSymbolFont{script}{U}{eus}{m}{n}
\DeclareMathSymbol{\Wedge}{0}{script}{"5E}
\begin{document}
\title{Calculus on symplectic manifolds}
\author[Michael Eastwood]{Michael Eastwood}
\address{\hskip-\parindent
School of Mathematical Sciences\\
University of Adelaide\\ 
SA 5005\\ 
Australia}
\email{meastwoo@member.ams.org}
\author[Jan Slov\'ak]{Jan Slov\'ak}
\address{\hskip-\parindent
Department of Mathematics and Statistics\\
Masaryk University,\newline
611 37 Brno, Czech Republic}
\email{slovak@math.muni.cz}
\thanks{This research was supported by the Czech Grant Agency. The
authors would like to thank the Agency for their generous support under Grant
P201/12/G028.}
\thanks{This work was also supported by the Simons Foundation grant 
346300 and the Polish Government MNiSW 2015--2019 matching fund. It was 
completed whilst the authors were visiting the Banach Centre at IMPAN in 
Warsaw for the Simons Semester `Symmetry and Geometric Stuctures.'}
\begin{abstract} On a symplectic manifold, there is a natural elliptic complex
replacing the de~Rham complex. It can be coupled to a vector bundle with
connection and, when the curvature of this connection is constrained to be a
multiple of the symplectic form, we find a new complex. In particular, on
complex projective space with its Fubini--Study form and connection, we can
build a series of differential complexes akin to the
Bernstein--Gelfand--Gelfand complexes from parabolic differential geometry.
\end{abstract}
\renewcommand{\subjclassname}{\textup{2010} Mathematics Subject Classification}
\maketitle
\section{Introduction}
Throughout this article $M$ will be a smooth manifold of dimension $2n$
equipped with a symplectic form $J_{ab}$. Here, we are using Penrose's abstract
index notation~\cite{OT} and non-degeneracy of this $2$-form says that there
is a skew contravariant $2$-form $J^{ab}$ such that $J_{ab}J^{ac}=\delta_b{}^c$
where $\delta_b{}^c$ is the canonical pairing between vectors and co-vectors.

Let $\Wedge^k$ denote the bundle of $k$-forms on~$M$. The homomorphism
$$\Wedge^k\to\Wedge^{k-2}\enskip\mbox{given by}\enskip
\phi_{abc\cdots d}\mapsto J^{ab}\phi_{abc\cdots d}$$
is surjective for $2\leq k\leq n$ with non-trivial kernel, corresponding to 
the irreducible representation 
$$\rule[-10pt]{20pt}{0pt}\begin{picture}(145,10)
\put(5,3){\line(1,0){20}}
\put(5,2.6){\makebox(0,0){$\bullet$}}
\put(20,2.6){\makebox(0,0){$\bullet$}}
\put(36,2.6){\makebox(0,0){$\cdots$}}
\put(50,2.6){\makebox(0,0){$\bullet$}}
\put(45,3){\line(1,0){40}}
\put(65,2.6){\makebox(0,0){$\bullet$}}
\put(80,2.6){\makebox(0,0){$\bullet$}}
\put(96,2.6){\makebox(0,0){$\cdots$}}
\put(105,3){\line(1,0){20}}
\put(110,2.6){\makebox(0,0){$\bullet$}}
\put(125,2.6){\makebox(0,0){$\bullet$}}
\put(125,5){\line(1,0){15}}
\put(125,1){\line(1,0){15}}
\put(140,2.6){\makebox(0,0){$\bullet$}}
\put(132.5,3){\makebox(0,0){$\langle$}}
\put(5,5){\makebox(0,0)[b]{\scriptsize$\vphantom{(}0$}}
\put(20,5){\makebox(0,0)[b]{\scriptsize$\vphantom{(}0$}}
\put(50,5){\makebox(0,0)[b]{\scriptsize$\vphantom{(}0$}}
\put(65,5){\makebox(0,0)[b]{\scriptsize$\vphantom{(}1$}}
\put(80,5){\makebox(0,0)[b]{\scriptsize$\vphantom{(}0$}}
\put(110,5){\makebox(0,0)[b]{\scriptsize$\vphantom{(}0$}}
\put(125,5){\makebox(0,0)[b]{\scriptsize$\vphantom{(}0$}}
\put(140,5){\makebox(0,0)[b]{\scriptsize$\vphantom{(}0$}}
\put(65,-10){\vector(0,1){8}}
\put(70,-6){\makebox(0,0)[l]{\scriptsize{$k^{\mathrm{th}}$ node}}}
\end{picture}\quad\mbox{of}\quad
{\mathrm{Sp}}(2n,{\mathbb{R}})\subset{\mathrm{GL}}(2n,{\mathbb{R}}).$$
Denoting this bundle by $\Wedge_\perp^k$, there is a canonical splitting of 
the short exact sequence
$$0\to\Wedge_\perp^k\raisebox{-6.3pt}{$\begin{array}{c}
\rightleftarrows\\[-8pt] \mbox{\scriptsize$\pi$}\end{array}$}
\Wedge^k\to\Wedge^{k-2}\to 0$$
and an elliptic complex~\cite{BEGN,E,ES,S,TY}
\begin{equation}
\label{RScomplex}\addtolength{\arraycolsep}{-1pt}\begin{array}{rcccccccccccc}
0&\to&\Wedge^0&\stackrel{d}{\longrightarrow}&\Wedge^1
&\stackrel{d_\perp}{\longrightarrow}&\Wedge_\perp^2
&\stackrel{d_\perp}{\longrightarrow}&\Wedge_\perp^3
&\stackrel{d_\perp}{\longrightarrow}&\cdots
&\stackrel{d_\perp}{\longrightarrow}&\Wedge_\perp^{n}\\[2pt]
&&&&&&&&&&&&\big\downarrow\makebox[0pt][l]{\scriptsize$d_\perp^2$}\\
0&\leftarrow&\Wedge^0&\stackrel{d_\perp}{\longleftarrow}&\Wedge^1
&\stackrel{d_\perp}{\longleftarrow}&\Wedge_\perp^2
&\stackrel{d_\perp}{\longleftarrow}&\Wedge_\perp^3
&\stackrel{d_\perp}{\longleftarrow}&\cdots
&\stackrel{d_\perp}{\longleftarrow}&\Wedge_\perp^{n}
\end{array}\end{equation}
where
\begin{itemize}
\item $d:\Wedge^0\to\Wedge^1$ is the exterior derivative,
\item for $1\leq k< n$, the operator
$d_\perp:\Wedge_\perp^k\to\Wedge_\perp^{k+1}$ is the composition
$$\Wedge_\perp^k\hookrightarrow\Wedge^k\xrightarrow{\,d\,}\Wedge^{k+1}
\xrightarrow{\,\pi\,}\Wedge_\perp^{k+1},$$
a first order operator,
\item $d_\perp:\Wedge_\perp^{k+1}\to\Wedge_\perp^k$ are canonically defined 
first order operators, which may be seen as adjoint to 
$d_\perp:\Wedge_\perp^k\to\Wedge_\perp^{k+1}$,
\item $d_\perp^2:\Wedge_\perp^n\to\Wedge_\perp^n$ is the composition
$$\Wedge_\perp^n\xrightarrow{\,d_\perp\,}\Wedge_\perp^{n-1}
\xrightarrow{\,d_\perp\,}\Wedge_\perp^n,$$
a second order operator.
\end{itemize}
More explicitly, formul{\ae} for these operators may be given as follows.
Firstly, it is convenient to choose a {\em symplectic connection\/}~$\nabla_a$,
namely a torsion-free connection such that $\nabla_aJ_{bc}=0$, equivalently
$\nabla_aJ^{bc}=0$. As shown in~\cite{GRS}, for example, such connections
always exist and if $\nabla_a$ is one such, then the general symplectic
connection is
$$\hat\nabla_a\phi_b=\nabla_a\phi_b+J^{cd}\Xi_{abc}\phi_d\quad\mbox{where}
\enskip\Xi_{abc}=\Xi_{(abc)}.$$
Then, for $1\leq k<n$, the operator
$d_\perp:\Wedge_\perp^k\to\Wedge_\perp^{k+1}$ 
is given by 
\begin{equation}\label{early}\textstyle\phi_{def\cdots g}\longmapsto
\nabla_{[c}\phi_{def\cdots g]}
-\frac{k}{2(n+1-k)}J^{ab}(\nabla_a\phi_{b[ef\cdots g})J_{cd]}\end{equation}
and $d_\perp:\Wedge_\perp^{k+1}\to\Wedge_\perp^k$ is given by
\begin{equation}\label{late}
\psi_{cdef\cdots g}\longmapsto J^{bc}\nabla_b\psi_{cdef\cdots g}.\end{equation}

Now suppose $E$ is a smooth vector bundle on $M$ and
$\nabla:E\to\Wedge^1\otimes E$ is a connection. Choosing any torsion-free
connection on $\Wedge^1$ induces a connection on $\Wedge^1\otimes E$ and, as is
well-known, the composition
$$\Wedge^1\otimes E\to
\Wedge^1\otimes\Wedge^1\otimes E\to\Wedge^2\otimes E$$
does not depend on this choice. (It is the second in a well-defined sequence of
differential operators
\begin{equation}\label{coupled_de_Rham}
E\xrightarrow{\,\nabla\,}\Wedge^1\otimes E
\xrightarrow{\,\nabla\,}\Wedge^2\otimes E
\xrightarrow{\,\nabla\,}\cdots
\xrightarrow{\,\nabla\,}\Wedge^{2n-1}\otimes E
\xrightarrow{\,\nabla\,}\Wedge^{2n}\otimes E\end{equation}
known as the {\em coupled de~Rham sequence\/}.)
In particular, we may define a homomorphism
$\Theta:E\to E$ by 
$$\textstyle J^{ab}\nabla_a\nabla_b\Sigma=\frac1{2n}\Theta\Sigma
\quad\mbox{for}\enskip\Sigma\in\Gamma(E).$$ 
It is part of the curvature of~$\nabla$ and if this is the only curvature, then
\begin{equation}\label{Theta_in_the_symplectically_flat_case}
(\nabla_a\nabla_b-\nabla_b\nabla_a)\Sigma=2J_{ab}\Theta\Sigma,\end{equation}
and we shall say that $\nabla$ is {\em symplectically flat\/}. Looking back at 
(\ref{RScomplex}), it is easy to see that there are coupled operators
$$E
\begin{array}c\scriptstyle\nabla\\[-8pt] \longrightarrow\\[-10pt] 
\longleftarrow\\[-9pt] \scriptstyle{}\enskip\nabla_\perp\end{array}
\Wedge^1\otimes E
\begin{array}c\scriptstyle\;\nabla\!{}_\perp\\[-8pt] \longrightarrow\\[-10pt] 
\longleftarrow\\[-9pt] \scriptstyle{}\;\nabla\!{}_\perp\end{array}
\Wedge_\perp^2\otimes E
\begin{array}c\scriptstyle\;\nabla\!{}_\perp\\[-8pt] \longrightarrow\\[-10pt] 
\longleftarrow\\[-9pt] \scriptstyle{}\;\nabla\!{}_\perp\end{array}
\cdots
\begin{array}c\scriptstyle\;\nabla\!{}_\perp\\[-8pt] \longrightarrow\\[-10pt] 
\longleftarrow\\[-9pt] \scriptstyle{}\;\nabla\!{}_\perp\end{array}
\Wedge_\perp^{n-1}\otimes E
\begin{array}c\scriptstyle\;\nabla\!{}_\perp\\[-8pt] \longrightarrow\\[-10pt] 
\longleftarrow\\[-9pt] \scriptstyle{}\;\nabla\!{}_\perp\end{array}
\Wedge_\perp^n\otimes E,$$
explicit formul\ae\ for which are just as in the uncoupled cases
(\ref{early}) and~(\ref{late}). 
To complete the coupled version of (\ref{RScomplex}) let us use
\begin{equation}\label{middle_operator}
\textstyle\nabla^2_\perp-\frac2n\Theta:
\Wedge_\perp^n\otimes E\longrightarrow\Wedge_\perp^n\otimes E\end{equation}
for the middle operator. It is evident that 
$$E\stackrel{\nabla}{\longrightarrow}\Wedge^1\otimes E
\xrightarrow{\,\nabla_\perp\,}\Wedge_\perp^2\otimes E$$
is a complex if and only if $\nabla$ is symplectically flat. The reason for the
curvature term in (\ref{middle_operator}) is that this feature propagates as
follows. 
\begin{thm}\label{one}
Suppose $E\xrightarrow{\,\nabla\,}\Wedge^1\otimes E$ is a
symplectically flat connection and define $\Theta:E\to E$ 
by~\eqref{Theta_in_the_symplectically_flat_case}.
Then the coupled version of \eqref{RScomplex}
$$\addtolength{\arraycolsep}{-1pt}\begin{array}{rcccccccccc}
0&\to&E&\stackrel{\nabla}{\longrightarrow}&\Wedge^1\otimes E
&\stackrel{\nabla_\perp}{\longrightarrow}&\Wedge_\perp^2\otimes E
&\stackrel{\nabla_\perp}{\longrightarrow}&\cdots
&\stackrel{\nabla_\perp}{\longrightarrow}&\Wedge_\perp^n\otimes E\\[2pt]
&&&&&&&&&&
\big\downarrow
\makebox[0pt][l]{\scriptsize$\nabla_\perp^2-\frac2{n}\Theta$}\\
0&\leftarrow&E&\stackrel{\nabla_\perp}{\longleftarrow}&\Wedge^1\otimes E
&\stackrel{\nabla_\perp}{\longleftarrow}&\Wedge_\perp^2\otimes E
&\stackrel{\nabla_\perp}{\longleftarrow}&\cdots
&\stackrel{\nabla_\perp}{\longleftarrow}&\Wedge_\perp^n\otimes E
\end{array}\quad$$
is a complex. It is locally exact except near the beginning where 
$$\ker\nabla:E\to\Wedge^1\otimes E\quad\mbox{and}\quad
\frac{\ker\nabla_\perp:\Wedge^1\otimes E\to\Wedge_\perp^2\otimes E}
{\im\nabla:E\to\Wedge^1\otimes E}$$
may be identified with the kernel and cokernel, respectively, of\/ $\Theta$ as
locally constant sheaves.
\end{thm}
\noindent 
More precision and a proof of Theorem~\ref{one} will be provided in 
\S\ref{rumin_seshadri}.
Our next theorem yields some natural symplectically flat connections.
\begin{thm}\label{two}
Suppose $M$ is a\/ $2n$-dimensional symplectic manifold with
symplectic connection~$\nabla_a$. Then there is a natural vector bundle\/
${\T}$ on $M$ of rank $2n+2$ equipped with a connection, which is 
symplectically flat if and only if the curvature $R_{ab}{}^c{}_d$ of 
$\nabla_a$ has the form
\begin{equation}\label{Vis0}
R_{ab}{}^c{}_d=\delta_a{}^c\Phi_{bd}-\delta_b{}^c\Phi_{ad}
+J_{ad}\Phi_{be}J^{ce}-J_{bd}\Phi_{ae}J^{ce}+2J_{ab}\Phi_{de}J^{ce},
\end{equation}
for some symmetric tensor~$\Phi_{ab}$.
\end{thm}
\noindent In particular, the Fubini--Study connection on complex projective
space is symplectic for the standard K\"ahler form and its curvature is of 
the form (\ref{Vis0}) for $\Phi_{ab}=g_{ab}$, the standard metric. 
More generally, if the symplectic connection $\nabla_a$ arises from a K\"ahler
metric, then we shall see that (\ref{Vis0}) holds precisely in the case of
constant holomorphic sectional curvature.

After proving Theorems~\ref{one} and~\ref{two}, the remainder of this article
is concerned with the consequences of Theorem~\ref{one} for the vector bundle
${\T}$ and those bundles, such as~$\bigodot^k\!{\T}$, induced
from it. In particular, these consequences pertain on complex projective space 
where we shall find a series of elliptic complexes closely following the 
Bernstein-Gelfand-Gelfand complexes on the sphere $S^{2n+1}$ as a homogeneous 
space for the Lie group ${\mathrm{Sp}}(2n+2,{\mathbb{R}})$.

This article is based on our earlier work~\cite{ES} but here we focus on the
simpler case where we are given a symplectic structure as background. This
results in fewer technicalities and in this article we include more detail,
especially in constructing the BGG-like complexes in~\S\ref{BGG-like}. 

\section{The Rumin--Seshadri complex}\label{rumin_seshadri}
By the {\em Rumin--Seshadri complex\/}, we mean the differential complex
(\ref{RScomplex}) after~\cite{S}. However, the $4$-dimensional case is due to
R.T.~Smith~\cite{Sm} and the general case is also independently due to Tseng
and Yau~\cite{TY}. In this section we shall derive the coupled version of this
complex as in Theorem~\ref{one}, our proof of which includes (\ref{RScomplex})
as a special case. The following lemma is also the key step in~\cite{ES}.
\begin{lemma}\label{key_lemma}
Suppose $E$ is a vector bundle on $M$ with symplectically flat
connection $\nabla:E\to\Wedge^1\otimes E$. Define $\Theta:E\to E$
by~\eqref{Theta_in_the_symplectically_flat_case}. Then $\Theta$ has constant
rank and the bundles $\ker\Theta$ and $\coker\Theta$ acquire from~$\nabla$,
flat connections defining locally constant sheaves \underbar{$\ker\Theta$} and
\underbar{$\coker\Theta$}, respectively. There is an elliptic complex
$$\begin{array}{cccccccccc}E
&\stackrel{\nabla}{\longrightarrow}&\Wedge^1\otimes E
&\stackrel{\nabla}{\longrightarrow}&\Wedge^2\otimes E
&\stackrel{\nabla}{\longrightarrow}&\Wedge^3\otimes E
&\stackrel{\nabla}{\longrightarrow}&\Wedge^4\otimes E\\
&\begin{picture}(0,0)(0,-3)
\put(-9,6){\vector(3,-2){18}}\end{picture}
&\oplus
&\begin{picture}(0,0)(0,-3)
\put(-9,-6){\vector(3,2){18}}
\put(-9,6){\vector(3,-2){18}}\end{picture}
&\oplus
&\begin{picture}(0,0)(0,-3)
\put(-9,-6){\vector(3,2){18}}
\put(-9,6){\vector(3,-2){18}}\end{picture}
&\oplus
&\begin{picture}(0,0)(0,-3)
\put(-9,-6){\vector(3,2){18}}
\put(-9,6){\vector(3,-2){18}}\end{picture}
&\oplus&\cdots,
\\
&&E&\longrightarrow&\Wedge^1\otimes E
&\longrightarrow&\Wedge^2\otimes E
&\longrightarrow&\Wedge^3\otimes E
\end{array}$$
where the differentials are given by 
$$\Sigma\!\mapsto\!\left[\!\begin{array}{c}\nabla\Sigma\\ 
\Theta\Sigma\end{array}\!\right]
\quad
\left[\!\begin{array}{c}\phi\\ \eta\end{array}\!\right]
\!\mapsto\!\left[\!\begin{array}{c}\nabla\phi-J\otimes\eta\\ 
\nabla\eta-\Theta\phi\end{array}\!\right]
\quad\left[\!\begin{array}{c}\omega\\ \psi\end{array}\!\right]
\!\mapsto\!\left[\!\begin{array}{c}\nabla\omega+J\wedge\psi\\ 
\nabla\psi+\Theta\omega\end{array}\!\right]\enskip\cdots.$$
It is locally exact save for the zeroth and first cohomologies, which may be
identified with \underbar{$\ker\Theta$} and \underbar{$\coker\Theta$}, 
respectively.
\end{lemma}
\begin{proof} {From} (\ref{Theta_in_the_symplectically_flat_case}) the Bianchi
identity for $\nabla$ reads
$$0=\nabla_{[a}\big(J_{bc]}\Theta\big)=J_{[ab}\nabla_{c]}\Theta$$
and non-degeneracy of $J_{ab}$ implies that $\nabla_a\Theta=0$. Consequently,
the homomorphism $\Theta$ has constant rank and the following diagram with 
exact rows commutes
$$\addtolength{\arraycolsep}{-2pt}\begin{array}{ccccccccccc}
0&\to&\ker\Theta&\to&E&\xrightarrow{\,\Theta\,}&E&\to&\coker\Theta&\to&0\\
&&&&\downarrow\!\makebox[0pt][l]{\scriptsize$\nabla$}
&&\downarrow\!\makebox[0pt][l]{\scriptsize$\nabla$}\\
0&\to&\Wedge^1\otimes\ker\Theta&\to&\Wedge^1\otimes E&\xrightarrow{\,\Theta\,}
&\Wedge^1\otimes E&\to&\Wedge^1\otimes\coker\Theta&\to&0
\end{array}$$
and yields the desired connections on $\ker\Theta$ and~$\coker\Theta$, which
are easily seen to be flat. Ellipticity of the given complex is readily
verified and, by definition, the kernel of its first differential 
is~\underbar{$\ker\Theta$}. To identify the higher local cohomology of this 
complex the key observation is that locally we may choose a $1$-form $\tau$ 
such that $d\tau=J$ and, having done this, the connection
$$\Gamma(E)\ni\Sigma\stackrel{\tilde\nabla}{\longmapsto}
\nabla\Sigma-\tau\otimes\Theta\Sigma\in
\Gamma(\Wedge^1\otimes E)$$
is flat. The rest of the proof is diagram chasing, using exactness of 
$$E\xrightarrow{\,\tilde\nabla\,}\Wedge^1\otimes E
\xrightarrow{\,\tilde\nabla\,}\Wedge^2\otimes E
\xrightarrow{\,\tilde\nabla\,}\Wedge^3\otimes E
\xrightarrow{\,\tilde\nabla\,}\Wedge^4\otimes E
\xrightarrow{\,\tilde\nabla\,}\cdots.$$
If needed, the details are in~\cite{ES}.
\end{proof}

\noindent{\em Proof of Theorem~\ref{one}}. In \cite{ES}, the corresponding 
result \cite[Theorem~4]{ES} is proved by invoking a spectral sequence. Here, 
we shall, instead, prove two typical cases `by hand,' leaving the rest of the 
proof to the reader. 

For our first case, let us suppose $n\geq 3$ and prove local exactness of
$$\Wedge^1\otimes E\xrightarrow{\,\nabla_\perp\,}\Wedge_\perp^2\otimes E
\xrightarrow{\,\nabla_\perp\,}\Wedge_\perp^3\otimes E.$$
Thus, we are required to show that if $\omega_{ab}$ has values in $E$ and
$$\textstyle\omega_{ab}=\omega_{[ab]}\qquad J^{ab}\omega_{ab}=0\qquad
\nabla_{[c}\omega_{de]}=\frac1{n-1}J^{ab}(\nabla_a\omega_{b[c})J_{de]},$$ 
then locally there is $\phi_{a}\in\Gamma(\Wedge^1\otimes E)$ such that
$$\textstyle\omega_{cd}
=\nabla_{[c}\phi_{d]}-\frac1{2n}J^{ab}(\nabla_a\phi_b)J_{cd}.$$
If we set $\psi_c\equiv-\frac1{n-1}J^{ab}\nabla_a\omega_{bc}$, then
$\nabla_{[c}\omega_{de]}+J_{[cd}\psi_{e]}=0$ so
$$0=\nabla_{[b}\nabla_c\omega_{de]}+J_{[bc}\nabla_d\psi_{e]}
=J_{[bc}\Theta\omega_{de]}+J_{[bc}\nabla_d\psi_{e]}$$
and since $J\wedge\underbar{\enskip}:\Wedge^2\to\Wedge^4$ is injective it 
follows that
$$\nabla_{[c}\psi_{d]}+\Theta\omega_{cd}=0.$$
In other words, we have shown that
$$\begin{array}{rcl}
\nabla\omega+J\wedge\psi&=&0\\
\nabla\psi+\Theta\omega&=&0\end{array}$$
and Lemma~\ref{key_lemma} locally yields
$\phi_a\in\Gamma(\Wedge^1\otimes E)$ and $\eta\in\Gamma(E)$ such that
$$\begin{array}{rcl}\nabla_{[a}\phi_{b]}-J_{ab}\eta&=&\omega_{ab}\\
\nabla_a\eta-\Theta\phi_a&=&\psi_a\end{array}$$
In particular, 
$$J^{ab}\nabla_a\phi_b-2n\eta=J^{ab}\big(\nabla_a\phi_b-J_{ab}\eta\big)
=J^{ab}\omega_{ab}=0$$
and, therefore,
$$\textstyle\nabla_{[c}\phi_{d]}-\frac1{2n}J^{ab}(\nabla_a\phi_b)J_{cd}
=\nabla_{[c}\phi_{d]}-\eta J_{cd}=\omega_{cd},$$
as required.

Our second case is more involved. It is to show that
\begin{equation}\label{complex}
\Wedge_\perp^n\otimes E\xrightarrow{\,\nabla_\perp^2-\frac2n\Theta\,}
\Wedge_\perp^n\otimes E\xrightarrow{\,\nabla_\perp\,}
\Wedge_\perp^{n-1}\otimes E\end{equation}
is locally exact. As regards
$\nabla_\perp:\Wedge_\perp^n\otimes E\to\Wedge_\perp^{n-1}\otimes E$,
notice that 
$$\textstyle J^{bc}\nabla_b\psi_{cdef\cdots g}
=\frac{n+1}2J^{bc}\nabla_{[b}\psi_{cdef\cdots g]}$$
and that if $\phi_{def\cdots g}\in\Gamma(\Wedge^{k}\otimes E)$, then
\begin{equation}\label{combinatorics}
\textstyle J^{bc}J_{[bc}\phi_{def\cdots g]}=
\frac{4(n-k)}{(k+1)(k+2)}\phi_{def\cdots g}+
\frac{k(k-1)}{(k+1)(k+2)}J_{[de}\phi_{f\cdots g]bc}J^{bc}\end{equation}
so if $\phi_{def\cdots g}\in\Gamma(\Wedge_\perp^{n-1}\otimes E)$, then
$$\textstyle J^{bc}J_{[bc}\phi_{def\cdots g]}=
\frac4{n(n+1)}\phi_{def\cdots g}.$$
Therefore, $\nabla_\perp\psi\in\Gamma(\Wedge_\perp^{n-1}\otimes E)$ is
characterised by
\begin{equation}\label{trick}
\textstyle J\wedge\nabla_\perp\psi=\frac2n\nabla\psi\end{equation}
as an equation in $\Wedge^{n+1}\otimes E$. In particular, in 
$\Wedge^{n+2}\otimes E$ we find
$$\textstyle
J\wedge\nabla\nabla_\perp\psi
=\nabla(J\wedge\nabla_\perp\psi)=\frac2n\nabla^2\psi
=J\wedge\Theta\psi=0$$
whence $\nabla\nabla_\perp\psi$ already lies in $\Wedge^n\otimes E$ and there
is no need to remove the trace as in (\ref{early}) to form
$\nabla_\perp^2\psi$. Therefore, invoking (\ref{trick}) once again, the
composition
$$\Wedge_\perp^n\otimes E\xrightarrow{\,\nabla_\perp\,}
\Wedge_\perp^{n-1}\otimes E
\xrightarrow{\,\nabla_\perp\,}\Wedge_\perp^n\otimes E
\xrightarrow{\,\nabla_\perp\,}\Wedge_\perp^{n-1}\otimes E$$
is characterised by
$$\textstyle J\wedge\nabla_\perp^3\psi=\frac2n\nabla\nabla_\perp^2\psi
=\frac2n\nabla^2\nabla_\perp\psi=\frac2nJ\wedge\Theta\nabla_\perp\psi
=\frac2nJ\wedge\nabla_\perp\Theta\psi$$
and, since $J\wedge\underbar{\enskip}:\Wedge^{n-1}\to\Wedge^{n+1}$ is an 
isomorphism, we conclude that 
$\nabla_\perp^3\psi=\frac2n\nabla_\perp\Theta\psi$, equivalently that 
(\ref{complex}) is a complex. 

Before proceeding, let us remark on another consequence of 
(\ref{combinatorics}), namely that for 
$\nu_{cdef\cdots g}\in\Gamma(\Wedge^n\otimes E)$,
\begin{equation}\label{algebra}
J_{[ab}\nu_{cdef\cdots g]}=0\iff
J^{cd}\nu_{cdef\cdots g}=0.
\end{equation}
Now to establish local exactness, suppose
$\nu\in\Gamma(\Wedge_\perp^n\otimes E)$ satisfies $\nabla_\perp\nu=0$.
Equivalently, according to (\ref{trick}) and~(\ref{algebra})
$$\nu\in\Gamma(\Wedge^n\otimes E)\quad\mbox{satisfies}\enskip 
\nabla\nu=0\enskip\mbox{and}\enskip J\wedge\nu=0.$$
Lemma~\ref{key_lemma} implies that locally there are 
$$\begin{array}{l}\phi\in\Gamma(\Wedge^n\otimes E)\\
\eta\in\Gamma(\Wedge^{n-1}\otimes E)\end{array}\enskip\mbox{such that}\enskip
\begin{array}{rcl}\nabla\phi-J\wedge\eta&=&0\\
\nabla\eta-\Theta\phi&=&\nu.\end{array}$$ 
Since 
$$0\to\Wedge^{n-2}\xrightarrow{\,J\wedge\underbar{\enskip}\,}\Wedge^n
\to\Wedge_\perp^n\to0$$
is exact, we can write $\phi$ uniquely as 
$$\phi=\psi+J\wedge\tau,$$ where 
$\psi\in\Gamma(\Wedge_\perp^n\otimes E)$ and 
$\tau\in\Gamma(\Wedge^{n-2}\otimes E)$. 
We conclude that
$$\begin{array}{rcl}\nabla\psi-J\wedge\hat\eta&=&0\\
\nabla\hat\eta-\Theta\psi&=&\nu,\end{array}\enskip\mbox{(where}\enskip
\hat\eta=\eta-\nabla\tau).$$ 
However, as discussed above, these equations say exactly that
$$\textstyle\nabla_\perp^2\psi-\frac2n\Theta\psi=\nu,$$
and exactness is shown.
%
%
\hfill$\square$

\section{Tractor bundles}\label{tractors}\label{tractor_bundles}
For the rest of the article we suppose that we are given, not only a manifold 
$M$ with symplectic form~$J_{ab}$, but also a torsion-free connection 
$\nabla_a$ on the tangent bundle (and hence on all other tensor bundles) such 
that $\nabla_aJ_{bc}=0$. This is sometimes called a {\em Fedosov 
structure}~\cite{GRS} on~$M$.
The curvature $R_{ab}{}^c{}_d$ of~$\nabla_a$, characterised by
$$(\nabla_a\nabla_b-\nabla_b\nabla_a)X^c=R_{ab}{}^c{}_dX^d,$$
satisfies
$$R_{ab}{}^c{}_d=R_{[ab]}{}^c{}_d\qquad R_{[ab}{}^c{}_{d]}=0\qquad
R_{ab}{}^c{}_dJ_{ce}=R_{ab}{}^c{}_eJ_{cd}$$
and enjoys the following decomposition into irreducible parts
$$R_{ab}{}^c{}_d=V_{ab}{}^c{}_d+\delta_a{}^c\Phi_{bd}-\delta_b{}^c\Phi_{ad}
+J_{ad}\Phi_{be}J^{ce}-J_{bd}\Phi_{ae}J^{ce}+2J_{ab}\Phi_{de}J^{ce},$$
for some symmetric~$\Phi_{ab}$, where $V_{ab}{}^a{}_d=0$ (reflecting 
the branching
$$\begin{picture}(15,10)
\put(0,0){\line(1,0){5}}
\put(0,5){\line(1,0){15}}
\put(0,10){\line(1,0){15}}
\put(0,0){\line(0,1){10}}
\put(5,0){\line(0,1){10}}
\put(10,5){\line(0,1){5}}
\put(15,5){\line(0,1){5}}
\end{picture}\enskip=\enskip\begin{picture}(17,10)
\put(0,0){\line(1,0){5}}
\put(0,5){\line(1,0){15}}
\put(0,10){\line(1,0){15}}
\put(0,0){\line(0,1){10}}
\put(5,0){\line(0,1){10}}
\put(10,5){\line(0,1){5}}
\put(15,5){\line(0,1){5}}
\put(17,3){\makebox(0,0){$\scriptstyle\perp$}}
\end{picture}\enskip\oplus\enskip\begin{picture}(10,5)
\put(0,0){\line(1,0){10}}
\put(0,5){\line(1,0){10}}
\put(0,0){\line(0,1){5}}
\put(5,0){\line(0,1){5}}
\put(10,0){\line(0,1){5}}
\end{picture}$$
of representations under
${\mathrm{GL}}(2n,{\mathbb{R}})\supset{\mathrm{Sp}}(2n,{\mathbb{R}})$).
Notice that
\begin{equation}\label{Phi}\textstyle\Phi_{bd}=\frac1{2(n+1)}R_{ab}{}^a{}_d
=\frac1{4(n+1)}J^{ae}R_{ae}{}^c{}_bJ_{cd}.\end{equation}
We define the {\em standard tractor bundle\/} to be the rank $2n+2$ vector
bundle ${\T}\equiv\Wedge^0\oplus\Wedge^1\oplus\Wedge^0$ with its 
{\em tractor connection\/}
$$\textstyle\nabla_a\!
\left[\!\begin{array}c\sigma\\ \mu_b\\ \rho\end{array}\!\right]=
\left[\!\begin{array}c\nabla_a\sigma-\mu_a\\
\nabla_a\mu_b+J_{ab}\rho+\Phi_{ab}\sigma\\ 
\nabla_a\rho-\Phi_{ab}J^{bc}\mu_c+S_a\sigma
\end{array}\!\right]\!,\enskip\mbox{where}\enskip
S_a\equiv\frac1{2n+1}J^{bc}\nabla_c\Phi_{ab}.$$
Readers familiar with conformal differential geometry may recognise the form of
this connection as following the tractor connection in that setting~\cite{BEG}.
If needs be, we shall write {\em symplectic tractor connection\/} to
distinguish the connection just defined from any alternatives. We shall need 
the following curvature identities.
\begin{lemma}\label{curvature_identities}
Let $Y_{abc}\equiv\frac1{2n+1}\nabla_cV_{ab}{}^c{}_d$. Then
\begin{equation}\label{contracted_Bianchi}
Y_{abc}=2\nabla_{[a}\Phi_{b]c}-2J_{c[a}S_{b]}+2J_{ab}S_c\end{equation}
and
\begin{equation}\label{nablaY}
\begin{array}{rcr}J^{ad}\nabla_aY_{bcd}
&=&J^{ad}V_{bc}{}^e{}_a\Phi_{ed}
+4n(J^{ad}\Phi_{ba}\Phi_{cd}-\nabla_{[b}S_{c]})\qquad\\[3pt]
&&{}+2J_{bc}J^{ad}(\nabla_aS_d-J^{ef}\Phi_{ae}\Phi_{df}).
\end{array}\end{equation}
\end{lemma}
\begin{proof} Writing the Bianchi identity $\nabla_{[a}R_{bc]}{}^d{}_e=0$ in 
terms of $V_{ab}{}^c{}_d$ and $\Phi_{ab}$ yields
$$\nabla_{[a}V_{bc]}{}^d{}_e=-2\delta_{[b}{}^d\nabla_a\Phi_{c]e}
+2J^{df}J_{e[b}\nabla_a\Phi_{c]f}-2J^{df}J_{[bc}\nabla_{a]}\Phi_{ef}.$$
and contracting over ${}_a{}^d$ gives
$$\begin{array}{rcr}
\frac13\nabla_aV_{bc}{}^a{}_e
&=&\frac{4(n-1)}3\nabla_{[b}\Phi_{c]e}
+\frac23\big[\nabla_{[b}\Phi_{c]e}-(2n+1)J_{e[b}S_{c]}\big]\qquad\\[3pt]
&&{}+\frac23\big[(2n+1)J_{bc}S_e+2\nabla_{[b}\Phi_{c]e}\big],
\end{array}$$
which is easily rearranged as~(\ref{contracted_Bianchi}).
For (\ref{nablaY}), firstly notice that
$$J^{ad}R_{ab}{}^e{}_d=J^{ed}R_{ab}{}^a{}_d=2(n+1)J^{ed}\Phi_{bd}$$
and the Bianchi symmetry may be written as 
$R_{a[b}{}^e{}_{c]}=-\frac12R_{bc}{}^e{}_a$. Thus,
$$\begin{array}{rcl}J^{ad}\nabla_a\nabla_b\Phi_{cd}
&\!\!=\!\!&\nabla_bJ^{ad}\nabla_a\Phi_{cd}
-J^{ad}R_{ab}{}^e{}_c\Phi_{ed}-J^{ad}R_{ab}{}^e{}_d\Phi_{ce}\\[3pt]
&\!\!=\!\!&-(2n+1)\nabla_bS_c
-J^{ad}R_{ab}{}^e{}_c\Phi_{ed}+2(n+1)J^{de}\Phi_{bd}\Phi_{ce}
\end{array}$$
and so
$$\textstyle J^{ad}\nabla_a\nabla_{[b}\Phi_{c]d}=-(2n+1)\nabla_{[b}S_{c]}
+\frac12J^{ad}R_{bc}{}^e{}_a\Phi_{ed}+2(n+1)J^{de}\Phi_{bd}\Phi_{ce}.$$
{From} (\ref{contracted_Bianchi}) we see that
$$J^{ad}\nabla_aY_{bcd}=2J^{ad}\nabla_a\nabla_{[b}\Phi_{c]d}
+2\nabla_{[b}S_{c]}+2J_{bc}J^{ad}\nabla_aS_d.$$
Therefore,
$$J^{ad}\nabla_aY_{bcd}
=J^{ad}R_{bc}{}^e{}_a\Phi_{ed}-4n\nabla_{[b}S_{c]}
+4(n+1)J^{de}\Phi_{bd}\Phi_{ce}+2J_{bc}J^{ad}\nabla_aS_d.$$
Finally, 
$$J^{ad}R_{bc}{}^e{}_a\Phi_{ed}
=J^{ad}V_{bc}{}^e{}_a\Phi_{ed}
-4J^{ad}\Phi_{ba}\Phi_{cd}
-2J_{bc}J^{ad}J^{ef}\Phi_{ae}\Phi_{df},$$
so
$$\begin{array}{rcl}J^{ad}\nabla_aY_{bcd}
&=&J^{ad}V_{bc}{}^e{}_a\Phi_{ed}
+4nJ^{ad}\Phi_{ba}\Phi_{cd}
-2J_{bc}J^{ad}J^{ef}\Phi_{ae}\Phi_{df}\\[3pt]
&&\quad{}-4n\nabla_{[b}S_{c]}
+2J_{bc}J^{ad}\nabla_aS_d,
\end{array}$$
which may be rearranged as~(\ref{nablaY}).
\end{proof}
\begin{prop}\label{tractor_curvature} The
tractor connection ${\T}\to\Wedge^1\otimes{\T}$ preserves the
non-degenerate skew form
$$\left\langle\left[\!\begin{array}c\sigma\\ \mu_b\\ \rho\end{array}\!\right],
\left[\!\begin{array}c\tilde\sigma\\ \tilde\mu_c\\ 
\tilde\rho\end{array}\!\right]\right\rangle\equiv
\sigma\tilde\rho+J^{bc}\mu_b\tilde\mu_c-\rho\tilde\sigma$$
and its curvature is given by
$$\setlength{\arraycolsep}{1pt}\begin{array}{rcl}
(\nabla_a\nabla_a-\nabla_b\nabla_a)\!\!
\left[\!\begin{array}c\sigma\\ \mu_d\\ \rho\end{array}\!\right]
&=&\left[\!\begin{array}{c}0\\
-V_{ab}{}^c{}_d\mu_c
+Y_{abd}\sigma\\ 
-Y_{abc}J^{cd}\mu_d
+\frac1{2n}(J^{cd}V_{ab}{}^e{}_c\Phi_{de}-J^{cd}\nabla_cY_{abd})\sigma
\end{array}\!\right]\\[20pt]
&&+2J_{ab}\!\left[\!\begin{array}{c}
\rho\\ J^{ce}\Phi_{cd}\mu_e-S_d\sigma\\ S_cJ^{cd}\mu_d
+\frac1{2n}J^{cd}(\nabla_cS_d-J^{ef}\Phi_{ce}\Phi_{df})\sigma
\end{array}\!\right]\!\!.
\end{array}$$
\end{prop}
\begin{proof}We expand
$$\left\langle
\nabla_a\!\left[\!\begin{array}c\sigma\\ \mu_b\\ \rho\end{array}\!\right],
\left[\!\begin{array}c\tilde\sigma\\ \tilde\mu_c\\ 
\tilde\rho\end{array}\!\right]\right\rangle+\left\langle
\left[\!\begin{array}c\sigma\\ \mu_b\\ \rho\end{array}\!\right],
\nabla_a\!\left[\!\begin{array}c\tilde\sigma\\ \tilde\mu_c\\ 
\tilde\rho\end{array}\!\right]\right\rangle$$
to obtain
$$\begin{array}{l}(\nabla_a\sigma-\mu_a)\tilde\rho
+\sigma(\nabla\tilde\rho-\Phi_{ab}J^{bc}\tilde\mu_c+S_a\tilde\sigma)\\
\enskip{}+J^{bc}(\nabla_a\mu_b+J_{ab}\rho+\Phi_{ab}\sigma)\tilde\mu_c
+J^{bc}\mu_b(\nabla_a\tilde\mu_c+J_{ac}\tilde\rho+\Phi_{ac}\tilde\sigma)\\
\quad{}-(\nabla_a\rho-\Phi_{ab}J^{bc}\mu_c+S_a\sigma)\tilde\sigma
-\rho(\nabla_a\tilde\sigma-\tilde\mu_a)\end{array}$$
in which all terms cancel save for 
$$(\nabla_a\sigma)\tilde\rho
+\sigma\nabla\tilde\rho
+J^{bc}(\nabla_a\mu_b)\tilde\mu_c
+J^{bc}\mu_b\nabla_a\tilde\mu_c
-(\nabla_a\rho)\tilde\sigma
-\rho\nabla_a\tilde\sigma,$$
which reduces to 
$$\nabla_a\big(\sigma\tilde\rho+J^{bc}\mu_b\tilde\mu_c-\rho\tilde\sigma\big),$$
as required. For the curvature, we readily compute
$$\nabla_{[a}\nabla_{b]}\!
\left[\!\begin{array}c\sigma\\ \mu_d\\ \rho\end{array}\!\right]
=\left[\!\begin{array}l\nabla_{[a}\nabla_{b]}\sigma-J_{ba}\rho\\
\nabla_{[a}\nabla_{b]}\mu_d
+J_{d[a}\Phi_{b]c}J^{ce}\mu_e
-\Phi_{d[a}\mu_{b]}
+T_{abd}\sigma\\ 
\nabla_{[a}\nabla_{b]}\rho-
T_{abc}J^{cd}\mu_d
+(\nabla_{[a}S_{b]}-J^{cd}\Phi_{ac}\Phi_{bd})\sigma
\end{array}\!\right],$$
where $T_{abc}\equiv\nabla_{[a}\Phi_{b]c}-J_{c[a}S_{b]}$. 
Lemma~\ref{curvature_identities}, 
however, states that
$$\textstyle T_{abc}=\frac12Y_{abc}-J_{ab}S_c$$
and 
$$\begin{array}{rcl}4n(\nabla_{[a}S_{b]}-J^{cd}\Phi_{ac}\Phi_{bd})
&=&J^{cd}V_{ab}{}^e{}_c\Phi_{de}-J^{cd}\nabla_cY_{abd}\\[3pt]
&&{}+2J_{ab}J^{cd}(\nabla_cS_d-J^{ef}\Phi_{ce}\Phi_{df}).
\end{array}$$
Therefore,
$$\begin{array}{rcl}\nabla_{[a}\nabla_{b]}\!
\left[\!\begin{array}c\sigma\\ \mu_d\\ \rho\end{array}\!\right]
&=&\left[\!\begin{array}{c}0\\
\nabla_{[a}\nabla_{b]}\mu_d
+J_{d[a}\Phi_{b]c}J^{ce}\mu_e
-\Phi_{d[a}\mu_{b]}
+\frac12Y_{abd}\sigma\\ 
-\frac12Y_{abc}J^{cd}\mu_d
+\frac1{4n}(J^{cd}V_{ab}{}^e{}_c\Phi_{de}-J^{cd}\nabla_cY_{abd})\sigma
\end{array}\!\right]\\[20pt]
&&\enskip{}+J_{ab}\!\left[\!\begin{array}{c}
\rho\\ -S_d\sigma\\ S_cJ^{cd}\mu_d
+\frac1{2n}J^{cd}(\nabla_cS_d-J^{ef}\Phi_{ce}\Phi_{df})\sigma
\end{array}\!\right].
\end{array}$$
Finally, 
$$R_{ab}{}^c{}_d\mu_c=V_{ab}{}^c{}_d\mu_c-2\Phi_{d[a}\mu_{b]}
+2J_{d[a}\Phi_{b]c}J^{ce}\mu_e
+2J_{ab}\Phi_{de}J^{ce}\mu_c,$$
so
$$\textstyle\nabla_{[a}\nabla_{b]}\mu_d
+J_{d[a}\Phi_{b]c}J^{ce}\mu_e
-\Phi_{d[a}\mu_{b]}
=-\frac12V_{ab}{}^c{}_d\mu_c-J_{ab}\Phi_{de}J^{ce}\mu_c$$
whence
$$\begin{array}{rcl}\nabla_{[a}\nabla_{b]}\!
\left[\!\begin{array}c\sigma\\ \mu_d\\ \rho\end{array}\!\right]
&=&\left[\!\begin{array}{c}0\\
-\frac12V_{ab}{}^c{}_d\mu_c
+\frac12Y_{abd}\sigma\\ 
-\frac12Y_{abc}J^{cd}\mu_d
+\frac1{4n}(J^{cd}V_{ab}{}^e{}_c\Phi_{de}-J^{cd}\nabla_cY_{abd})\sigma
\end{array}\!\right]\\[20pt]
&&\enskip{}+J_{ab}\!\left[\!\begin{array}{c}
\rho\\ J^{ce}\Phi_{cd}\mu_e-S_d\sigma\\ S_cJ^{cd}\mu_d
+\frac1{2n}J^{cd}(\nabla_cS_d-J^{ef}\Phi_{ce}\Phi_{df})\sigma
\end{array}\!\right],
\end{array}$$
as required.
\end{proof}

\begin{cor}\label{symplectically_flat_tractors}
The tractor connection is symplectically flat if and only if the curvature
tensor $V_{ab}{}^c{}_d$ vanishes.
\end{cor}

\section{K\"ahler geometry}
K\"ahler manifolds provide a familiar source of symplectic manifolds equipped
with a compatible torsion-free connection as in~\S\ref{tractors}. In this case,
the connection $\nabla_a$ is the Levi-Civita connection of a metric $g_{ab}$
and $J_a{}^b\equiv J_{ac}g^{bc}$ is an almost complex structure on~$M$ whose
integrability is equivalent to the vanishing of~$\nabla_aJ_{bc}$. In K\"ahler
geometry, the Riemann curvature tensor decomposes into three irreducible parts:
\begin{equation}\label{Bochner_in_real_money}
\begin{array}{l}
R_{ab}{}^c{}_d=U_{ab}{}^c{}_d\\
\enskip{}+\delta_a{}^c\Xi_{bd}-\delta_b{}^c\Xi_{ad}
-g_{ad}\Xi_b{}^c+g_{bd}\Xi_a{}^c\\
\quad{}+J_a{}^c\Sigma_{bd}
-J_b{}^c\Sigma_{ad}
-J_{ad}\Sigma_b{}^c
+J_{bd}\Sigma_a{}^c
+2J_{ab}\Sigma^c{}_d
+2J^c{}_d\Sigma_{ab}\\
\enskip\quad{}+\Lambda(\delta_a{}^cg_{bd}-\delta_b{}^cg_{ad}
+J_a{}^cJ_{bd}
-J_b{}^cJ_{ad}
+2J_{ab}J^c{}_d),
\end{array}\end{equation}
where indices have been raised using $g^{ab}$ and
\begin{itemize}
\item $U_{ab}{}^c{}_d$ is totally trace-free with respect to
$g^{ab}$, $J_a{}^b$, and $J^{ab}$,
\item $\Xi_{ab}$ is trace-free symmetric whilst 
$\Sigma_{ab}\equiv J_a{}^c\Xi_{bc}$ is skew.
\end{itemize}
Computing the Ricci curvature from this decomposition, we find
$$R_{bd}\equiv R_{ab}{}^a{}_d=2(n+2)\Xi_{bd}+2(n+1)\Lambda g_{bd}$$
and therefore from (\ref{Phi}) conclude that
$$\textstyle\Phi_{ab}=\frac{n+2}{n+1}\Xi_{ab}+\Lambda g_{ab}.$$
Hence
$$\begin{array}{rcl}J_c{}^aR_{ab}{}^c{}_d
&=&J_c{}^aV_{ab}{}^c{}_d
-J_{bd}\Phi_a{}^a
-2J_b{}^a\Phi_{da}\\[3pt]
&=&J_c{}^aV_{ab}{}^c{}_d
-2\frac{n+2}{n+1}\Sigma_{bd}
-2(n+1)\Lambda J_{bd}.\end{array}$$
On the other hand, from (\ref{Bochner_in_real_money}) we find
$$J_c{}^aR_{ab}{}^c{}_d=
-2(n+2)\Sigma_{bd}-2(n+1)\Lambda J_{bd}$$
and, comparing these two expressions gives
$$\textstyle 
J_c{}^aV_{ab}{}^c{}_d-2\frac{n+2}{n+1}\Sigma_{bd}=-2(n+2)\Sigma_{bd}$$
and we have established the following.
\begin{prop}\label{Kaehler_consequence}
Concerning the symplectic curvature decomposition on a K\"ahler manifold,
$$\textstyle 
J_c{}^aV_{ab}{}^c{}_d
=-2\frac{n(n+2)}{n+1}\Sigma_{bd}.$$
\end{prop}
\begin{cor}\label{whenKaehlerVvanishes} The symplectic tractor connection on a
K\"ahler manifold is symplectically flat if and only if the metric has constant
holomorphic sectional curvature.
\end{cor}
\begin{proof} According to Corollary~\ref{symplectically_flat_tractors}, we
have to interpret the constraint $V_{ab}{}^c{}_d=0$ in the K\"ahler case.
{From} (\ref{Bochner_in_real_money}) it is already clear that
$U_{ab}{}^c{}_d=0$ and Proposition~\ref{Kaehler_consequence} implies that also
$\Sigma_{ab}=0$ so (\ref{Bochner_in_real_money}) reduces to
$$R_{ab}{}^c{}_d=\Lambda(\delta_a{}^cg_{bd}-\delta_b{}^cg_{ad}
+J_a{}^cJ_{bd}
-J_b{}^cJ_{ad}
+2J_{ab}J^c{}_d),$$
which is exactly the constancy of holomorphic sectional curvature.
\end{proof}

\section{BGG-like complexes on ${\mathbb{CP}}_n$}\label{BGG-like}
Fix a real vector space ${\mathfrak{g}}_{-1}$ of dimension $2n$, let 
${\mathfrak{g}}_1$ denotes its dual, and fix a non-degenerate $2$-form 
$J_{ab}\in\Wedge^2{\mathfrak{g}}_1$. The $(2n+1)$-dimensional Heisenberg Lie 
algebra may be realised as 
$${\mathfrak{h}}={\mathbb{R}}\oplus{\mathfrak{g}}_{-1},$$ 
where the first summand is the $1$-dimensional centre of ${\mathfrak{h}}$ and
the Lie bracket on ${\mathfrak{g}}_{-1}$ is given by
$$[X,Y]=2J_{ab}X^aY^b\in{\mathbb{R}}\hookrightarrow{\mathfrak{h}}.$$
We should admit right away that the reason for this seemingly arcane notation
is that we shall soon have occasion to write
\begin{equation}\label{sp(2n+2,R)}
\addtolength{\arraycolsep}{-2pt}{\mathfrak{sp}}(2n+2,{\mathbb{R}})
=\begin{array}[t]{ccccccccc}{\mathfrak{g}}_{-2}&\oplus&{\mathfrak{g}}_{-1}
&\oplus&{\mathfrak{g}}_0&\oplus
&{\mathfrak{g}}_1&\oplus&{\mathfrak{g}}_2\\
\|&&&&\|&&&&\|\\
{\mathbb{R}}
&&&&\makebox[0pt]{${\mathfrak{sp}}(2n,{\mathbb{R}})\oplus{\mathbb{R}}$\quad}
&&&&{\mathbb{R}}\end{array}\end{equation}
(a $|2|$-graded Lie algebra as in \cite[\S4.2.6]{CS}) and, in particular,
regard ${\mathfrak{h}}={\mathbb{R}}\oplus{\mathfrak{g}}_{-1}
={\mathfrak{g}}_{-2}\oplus{\mathfrak{g}}_{-1}$ as a Lie subalgebra of
${\mathfrak{sp}}(2n+2,{\mathbb{R}})$. Be that as it may, let us suppose that
${\mathbb{V}}$ is a finite-dimensional representation of~${\mathfrak{h}}$. The
Lie algebra cohomology $H^r({\mathfrak{h}},{\mathbb{V}})$ may be realised as
the cohomology of the Chevalley-Eilenberg complex
\begin{equation}\label{C-E}
0\to{\mathbb{V}}\to{\mathfrak{h}}^*\otimes{\mathbb{V}}\to
\cdots\to\Wedge^r{\mathfrak{h}}^*\otimes{\mathbb{V}}\to
\Wedge^{r+1}{\mathfrak{h}}^*\otimes{\mathbb{V}}\to\cdots\end{equation}
as, for example, in~\cite[Chapter~IV]{Knapp}. We shall require, however, the
following alternative realisation. 
\begin{lemma}\label{LieAlgBGG} There is a complex
\begin{equation}\addtolength{\arraycolsep}{-2pt}\label{HeisenbergBGG}
\begin{array}{rcccccccccc}
0&\to&{\mathbb{V}}&\stackrel{\partial}{\longrightarrow}&
{\mathfrak{g}}_1\otimes{\mathbb{V}}
&\stackrel{\partial_\perp}{\longrightarrow}&
\Wedge_\perp^2{\mathfrak{g}}_1\otimes{\mathbb{V}}
&\stackrel{\partial_\perp}{\longrightarrow}&\cdots
&\stackrel{\partial_\perp}{\longrightarrow}&
\Wedge_\perp^n{\mathfrak{g}}_1\otimes{\mathbb{V}}\\[2pt]
&&&&&&&&&&
\big\downarrow\\
0&\leftarrow&{\mathbb{V}}&\stackrel{\partial_\perp}{\longleftarrow}&
{\mathfrak{g}}_1\otimes{\mathbb{V}}
&\stackrel{\partial_\perp}{\longleftarrow}&
\Wedge_\perp^2{\mathfrak{g}}_1\otimes{\mathbb{V}}
&\stackrel{\partial_\perp}{\longleftarrow}&\cdots
&\stackrel{\partial_\perp}{\longleftarrow}&
\Wedge_\perp^n{\mathfrak{g}}_1\otimes{\mathbb{V}}
\end{array}\end{equation}
whose cohomology realises $H^r({\mathfrak{h}},{\mathbb{V}})$. Here, we are
writing
$$\Wedge_\perp^r{\mathfrak{g}}_1
\equiv\{\omega_{abc\cdots d}\in\Wedge^r{\mathfrak{g}}_1\mid
J^{ab}\omega_{abc\cdots d}=0\},$$
where $J^{ab}\in\Wedge^2{\mathfrak{g}}_{-1}$ is the
inverse of $J_{ab}\in\Wedge^2{\mathfrak{g}}_1$ 
(let's say normalised so that $J_{ab}J^{ac}=\delta_b{}^c$). 
\end{lemma}
\begin{proof}
Notice that any representation $\rho:{\mathfrak{h}}\to\End({\mathbb{V}})$ is
determined by its restriction to~${\mathfrak{g}}_{-1}\subset{\mathfrak{h}}$.
Indeed, writing $\partial_a:{\mathfrak{g}}_{-1}\to\End({\mathbb{V}})$ for this
restriction, to say that $\rho$ is a representation of ${\mathfrak{h}}$ is to
say that
\begin{equation}\label{rep_of_h}\begin{array}{rcl}
(\partial_a\partial_b-\partial_b\partial_a)v
&=&2J_{ab}\theta v\\
(\partial_a\theta-\theta\partial_a)v&=&0\end{array}
\raisebox{2pt}{$\Big\}\quad\forall\,v\in{\mathbb{V}}$},
\end{equation}
where $\theta\in\End({\mathbb{V}})$ is $\rho(1)$ for 
$1\in{\mathbb{R}}\subset{\mathfrak{h}}$. 

The splitting ${\mathfrak{h}}^*={\mathfrak{g}}_1\oplus{\mathbb{R}}$
allows us to write (\ref{C-E}) as 
\begin{equation}\label{C-E_rewritten}
\addtolength{\arraycolsep}{-2pt}\begin{array}{cccccccccc}
{\mathbb{V}}&\longrightarrow&{\mathfrak{h}}^*\otimes{\mathbb{V}}
&\longrightarrow&\Wedge^2{\mathfrak{h}}^*\otimes{\mathbb{V}}
&\longrightarrow&\Wedge^3{\mathfrak{h}}^*\otimes{\mathbb{V}}
&\longrightarrow&\cdots\\
\|&&\|&&\|&&\|\\
{\mathbb{V}}&\longrightarrow&{\mathfrak{g}}_1\otimes{\mathbb{V}}
&\longrightarrow&\Wedge^2{\mathfrak{g}}_1\otimes {\mathbb{V}}
&\longrightarrow&\Wedge^3{\mathfrak{g}}_1\otimes {\mathbb{V}}
&\longrightarrow&\cdots\,,\\
&\begin{picture}(0,0)(0,-3)
\put(-9,6){\vector(3,-2){18}}\end{picture}
&\oplus
&\begin{picture}(0,0)(0,-3)
\put(-9,-6){\vector(3,2){18}}
\put(-9,6){\vector(3,-2){18}}\end{picture}
&\oplus
&\begin{picture}(0,0)(0,-3)
\put(-9,-6){\vector(3,2){18}}
\put(-9,6){\vector(3,-2){18}}\end{picture}
&\oplus
&\begin{picture}(0,0)(0,-3)
\put(-9,-6){\vector(3,2){18}}
\put(-9,6){\vector(3,-2){18}}\end{picture}\\
&&{\mathbb{V}}&\longrightarrow
&{\mathfrak{g}}_1\otimes{\mathbb{V}}
&\longrightarrow
&\Wedge^2{\mathfrak{g}}_1\otimes{\mathbb{V}}
&\longrightarrow
&\cdots
\end{array}\end{equation}
where the differentials are given by 
$$v\!\mapsto\!\left[\!\begin{array}{c}\partial_av\\ 
\theta v\end{array}\!\right]
\quad
\left[\!\begin{array}{c}\phi_a\\ \eta\end{array}\!\right]
\!\mapsto\!\left[\!\begin{array}{c}\partial_{[a}\phi_{b]}-J_{ab}\eta\\ 
\partial_a\eta-\theta\phi_a\end{array}\!\right]
\quad\left[\!\begin{array}{c}\omega_{ab}\\ \psi_a\end{array}\!\right]
\!\mapsto\!
\left[\!\begin{array}{c}\partial_{[a}\omega_{bc]}+J_{[ab}\psi_{c]}\\ 
\partial_{[a}\psi_{b]}+\theta\omega_{ab}\end{array}\!\right]$$
et cetera. In particular, notice that the homomorphisms 
\begin{equation}\label{Jwedge}
\Wedge^{r-1}{\mathfrak{g}}_1\ni\psi\longmapsto
\pm J\wedge\psi\in\Wedge^{r+1}{\mathfrak{g}}_1\end{equation}
are
\begin{itemize}
\item independent of the representation on~${\mathbb{V}}$,
\item injective for $1\leq r<n$,
\item an isomorphism for $r=n$,
\item surjective for $n<r\leq 2n-1$.
\end{itemize}
Note that $\Wedge_\perp^{r+1}{\mathfrak{g}}_1$ is complementary to the image 
of (\ref{Jwedge}) for $1\leq r<n$. Also note the isomorphisms
$$\Wedge^{2n+1-r}{\mathfrak{g}}_1
\xrightarrow{\,J\wedge J\wedge\cdots\wedge J\,}
\Wedge^{r-1}{\mathfrak{g}}_1,\quad\mbox{for}\enskip
n<r\leq 2n+1,$$
under which the kernel of (\ref{Jwedge}) 
may be identified with 
$$\Wedge^{2n+1-r}_\perp{\mathfrak{g}}_1,
\quad\mbox{for}\enskip n<r\leq 2n-1.$$ 
Diagram chasing in (\ref{C-E_rewritten}) (or the spectral sequence of a 
filtered complex) finishes the proof.
\end{proof}
\noindent{\bf Remark.} Evidently, the equations (\ref{rep_of_h}) are algebraic 
versions of 
$$\begin{array}{rcl}
(\nabla_a\nabla_b-\nabla_b\nabla_a)\Sigma
&=&2J_{ab}\Theta \Sigma\\
(\nabla_a\Theta-\Theta\nabla_a)\Sigma&=&0\end{array}
\raisebox{2pt}{$\Big\}\quad\forall\,\Sigma\in\Gamma(E)$},$$
which hold for a symplectically flat connection $\nabla_a$ on smooth vector
bundle $E$ on~$M$. Also (\ref{C-E_rewritten}) is the evident algebraic
counterpart to the differential complex of Lemma~\ref{key_lemma}. It follows
that explicit formul{\ae} for the operators $\partial_\perp$ in the complex
(\ref{HeisenbergBGG}) follow the differential versions (\ref{early}) and
(\ref{late}) with $\Wedge_\perp^n{\mathfrak{g}}\otimes{\mathbb{V}}\to
\Wedge_\perp^n{\mathfrak{g}}\otimes{\mathbb{V}}$ being given by
$\partial_\perp^2-\frac2n\theta$.


Let us now consider the tractor connection on ${\mathbb{CP}}_n$. According to
Theorem~\ref{two}, the remarks following its statement, and the discussions
in~\S\ref{tractor_bundles}, this is the connection on 
${\T}=\Wedge^0\oplus\Wedge^1\oplus\Wedge^0$ 
given by
$$\nabla_a\!\left[\!\begin{array}c\sigma\\ \mu_b\\ 
\rho\end{array}\!\right]=
\left[\!\begin{array}c\nabla_a\sigma-\mu_a\\
\nabla_a\mu_b+J_{ab}\rho+g_{ab}\sigma\\ 
\nabla_a\rho-J_a{}^b\mu_b
\end{array}\!\right]=
\left[\!\begin{array}c\nabla_a\sigma\\
\nabla_a\mu_b+g_{ab}\sigma\\ 
\nabla_a\rho-J_a{}^b\mu_b
\end{array}\!\right]
+\left[\!\begin{array}c-\mu_a\\
J_{ab}\rho\\ 
0\end{array}\!\right]\!.$$
The induced operator
$\nabla:\Wedge^1\otimes{\T}\to\Wedge^2\otimes{\T}$ is
$$\left[\!\begin{array}c\sigma_b\\ \mu_{bc}\\ 
\rho_b\end{array}\!\right]\longmapsto
\left[\!\begin{array}c\nabla_{[a}\sigma_{b]}\\
\nabla_{[a}\mu_{b]c}+g_{c[a}\sigma_{b]}\\ 
\nabla_{[a}\rho_{b]}-J_{[a}{}^c\mu_{b]c}
\end{array}\!\right]
+\left[\!\begin{array}c\mu_{[ab]}\\
-J_{c[a}\rho_{b]}\\ 
0\end{array}\!\right]\!$$
but Corollary~\ref{whenKaehlerVvanishes} says the tractor connection
on ${\mathbb{CP}}_n$ is symplectically flat so we should contemplate
$\nabla_\perp:\Wedge^1\otimes{\T}\to\Wedge_\perp^2\otimes{\T}$ from
Theorem~\ref{one}, viz.
$$\left[\!\begin{array}c\sigma_b\\ \mu_{bc}\\ 
\rho_b\end{array}\!\right]\longmapsto
\left[\!\begin{array}c
\nabla_{[a}\sigma_{b]}-\frac1{2n}J^{cd}\nabla_c\sigma_dJ_{ab}\\
\ldots\\ 
\ldots\end{array}\!\right]
+\left[\!\begin{array}c\mu_{[ab]}-\frac1{2n}J^{cd}\mu_{cd}J_{ab}\\
-J_{c[a}\rho_{b]}-\frac1{2n}\rho_cJ_{ab}\\ 
0\end{array}\!\right]\!.$$
From these formul{\ae}, let us focus attention on the homomorphisms
\begin{equation}\label{attention_is_focussed}
\makebox[0pt]{$\begin{array}{ccccccccc}
0&\to&{\T}&\to&\Wedge^1\otimes{\T}&\to
&\Wedge_\perp^2\otimes{\T}&\to&\cdots\\[4pt]
&&\left[\!\begin{array}c\sigma\\ \mu_b\\ 
\rho\end{array}\!\right]&\mapsto
&\left[\!\begin{array}c-\mu_a\\
J_{ab}\rho\\ 
0\end{array}\!\right]\\[22pt]
&&&&\left[\!\begin{array}c\sigma_b\\ \mu_{bc}\\ 
\rho_b\end{array}\!\right]
&\mapsto
&\left[\!\begin{array}c\mu_{[ab]}-\frac1{2n}J^{cd}\mu_{cd}J_{ab}\\
-J_{c[a}\rho_{b]}-\frac1{2n}\rho_cJ_{ab}\\ 
0\end{array}\!\right]
\end{array}$}\end{equation}
It is evident that this is a complex and that its cohomology so far is 
$$\textstyle\Wedge^0\mbox{ in degree }0\quad\mbox{and}\quad
\bigodot^2\!\Wedge^1\mbox{ in degree }1.$$
On the other hand, one may check that the defining representation of the Lie
algebra ${\mathfrak{sp}}(2n+2,{\mathbb{R}})$ on
${\mathbb{R}}^{2n+2}={\mathbb{R}}\oplus{\mathbb{R}}^{2n}\oplus{\mathbb{R}}$
restricts via (\ref{sp(2n+2,R)}) to a representation of the Heisenberg Lie
algebra ${\mathfrak{h}}={\mathbb{R}}\oplus{\mathfrak{g}}_{-1}$, given 
explicitly by
$$\begin{array}[t]{ccc}
{\mathbb{R}}^{2n+2}&\xrightarrow{\,\theta\,}
&{\mathbb{R}}^{2n+2}\\
\left[\!\begin{array}c\sigma\\ \mu_b\\ 
\rho\end{array}\!\right]&\longmapsto&\left[\!\begin{array}c\rho\\
0\\ 
0\end{array}\!\right]
\end{array}\mbox{\quad and\quad}\begin{array}[t]{ccc}
{\mathbb{R}}^{2n+2}&\xrightarrow{\,\partial_a\,}
&{\mathfrak{g}}_{1}\otimes{\mathbb{R}}^{2n+2}\\
\left[\!\begin{array}c\sigma\\ \mu_b\\ 
\rho\end{array}\!\right]&\longmapsto&\left[\!\begin{array}c-\mu_a\\
J_{ab}\rho\\ 
0\end{array}\!\right]
\end{array}$$
(noticing that equations (\ref{rep_of_h}) hold, as they must). We may also find
$\theta$ as part of the curvature of the tractor connection
on~${\mathbb{CP}}_n$. Specifically, the formula from
Proposition~\ref{tractor_curvature} reduces to
\begin{equation}\label{tractor_curvature_on_CPn}
(\nabla_a\nabla_a-\nabla_b\nabla_a)\!
\left[\!\begin{array}c\sigma\\ \mu_d\\ \rho\end{array}\!\right]
=2J_{ab}\!\left[\!\begin{array}{c}
\rho\\ J_d{}^e\mu_e\\ 
-\sigma
\end{array}\!\right]\end{equation}
and we find $\theta$ as the top component of
$\Theta:{\T}\to{\T}$ where $\Theta$ is defined 
by~(\ref{Theta_in_the_symplectically_flat_case}). If we now consider the 
entire complex from Theorem~\ref{one}, with filtration induced by
$$\begin{array}{ccccccc}\Wedge^0&\subset&\Wedge^1\oplus\Wedge^0
&\subset&\Wedge^0\oplus\Wedge^1\oplus\Wedge^0&=&{\T}\\
\left[\!\begin{array}c 0\\ 0\\ \rho\end{array}\!\right]
&&\left[\!\begin{array}c 0\\ \mu_b\\ \rho\end{array}\!\right]
&&\left[\!\begin{array}c\sigma\\ \mu_b\\ \rho\end{array}\!\right]
\end{array}$$
of ${\T}$, then the associated spectral sequence (or corresponding
diagram chasing) yields (\ref{attention_is_focussed}) continuing as in
(\ref{HeisenbergBGG}) including the middle operator
$\nabla_\perp^2-\frac2n\theta:\Wedge_\perp^n\to\Wedge_\perp^n$. The same
reasoning pertains for any Fedosov structure with $V_{ab}{}^c{}_d=0$ as in
Corollary~\ref{symplectically_flat_tractors}. Evidently, this sequence of
vector bundle homomorphisms is induced by the complex (\ref{HeisenbergBGG})
and, together with Lemma~\ref{LieAlgBGG}, the spectral sequence of a filtered
complex (or the appropriate diagram chasing) immediately yields the following.
\begin{thm}\label{towardsBGG}
Suppose $\nabla_a$ is a torsion-free connection on a symplectic manifold
$(M,J_{ab})$, such that $\nabla_aJ_{bc}=0$ and so that the corresponding
curvature tensor $V_{ab}{}^c{}_d$ vanishes. Fix a finite-dimensional
representation ${\mathbb{E}}$ of\/ ${\mathrm{Sp}}(2n+2,{\mathbb{R}})$ and let
$E$ denote the associated `tractor bundle' induced from the standard tractor
bundle and the representation~${\mathbb{E}}$ (so that the standard
representation of\/ ${\mathrm{Sp}}(2n+2,{\mathbb{R}})$ on\/
${\mathbb{R}}^{2n+2}$ yields the standard tractor bundle). In accordance with
Corollary~\ref{symplectically_flat_tractors}, the induced `tractor connection'
$\nabla:E\to\Wedge^1\otimes E$ is symplectically flat and we may define
$\Theta:E\to E$ by~\eqref{Theta_in_the_symplectically_flat_case}. Having done 
this, there are complexes of differential operators
$$\addtolength{\arraycolsep}{-3pt}\begin{array}{rcccccccccc}
0&\to&H^0({\mathfrak{h}},E)&\to&H^1({\mathfrak{h}},E)
&\to&H^2({\mathfrak{h}},E)
&\to&\cdots
&\to&H^n({\mathfrak{h}},E)\\[2pt]
&&&&&&&&&&
\big\downarrow\\
0&\leftarrow&H^{2n+1}({\mathfrak{h}},E)&\leftarrow
&H^{2n}({\mathfrak{h}},E)&\leftarrow
&H^{2n-1}({\mathfrak{h}},E)
&\leftarrow&\cdots
&\leftarrow&H^{n+1}({\mathfrak{h}},E)
\end{array}$$
where $H^r({\mathfrak{h}},E)$ denotes the tensor bundle on $M$ that is induced
by the cohomology $H^r({\mathfrak{h}},{\mathbb{E}})$ as a representation of
${\mathrm{Sp}}(2n,{\mathbb{R}})$. This complex is 
locally exact except near the beginning where 
$$\ker:H^0({\mathfrak{h}},E)\to H^1({\mathfrak{h}},E)
\quad\mbox{and}\quad
\frac{\ker:H^1({\mathfrak{h}},E)\to H^2({\mathfrak{h}},E)}
{\im:H^0({\mathfrak{h}},E)\to H^1({\mathfrak{h}},E)}$$
may be identified with the locally constant sheaves \underbar{$\ker\Theta$} and
\underbar{$\coker\Theta$}, respectively. In particular, for\/
${\mathbb{CP}}_n$ with its Fubini--Study connection, these sheaves vanish and 
the complex is locally exact everywhere. 
\end{thm}
\begin{proof} It remains only to observe that for the Fubini--Study connection
we see from (\ref{tractor_curvature_on_CPn}) that
$\Theta:{\T}\to{\T}$ is an isomorphism. \end{proof}
\noindent The main point about Theorem~\ref{towardsBGG}, however, is that if
the representation ${\mathbb{E}}$ of ${\mathrm{Sp}}(2n+2,{\mathbb{R}})$ is
irreducible, then the representations $H^r({\mathfrak{h}},{\mathbb{E}})$ of
${\mathrm{Sp}}(2n,{\mathbb{R}})$ are also irreducible and are computed by a
theorem due to Kostant~\cite{K}. Specifically, if we denote the irreducible
representations of ${\mathrm{Sp}}(2n+2,{\mathbb{R}})$ and
${\mathrm{Sp}}(2n,{\mathbb{R}})$ by writing the highest weight as a linear
combination of fundamental weights and recording the coefficients over the
corresponding nodes of the Dynkin diagrams for $C_{n+1}$ and $C_n$, as is often
done, then Kostant's Theorem says that
$$\begin{array}{ccl}H^0({\mathfrak{h}},\begin{picture}(84,5)
\put(4,1.5){\line(1,0){42}}
\put(4,1.2){\makebox(0,0){$\bullet$}}
\put(16,1.2){\makebox(0,0){$\bullet$}}
\put(28,1.2){\makebox(0,0){$\bullet$}}
\put(40,1.2){\makebox(0,0){$\bullet$}}
\put(55,1.2){\makebox(0,0){$\cdots$}}
\put(62,1.5){\line(1,0){6}}
\put(68,1.2){\makebox(0,0){$\bullet$}}
\put(68,0.5){\line(1,0){12}}
\put(68,2.5){\line(1,0){12}}
\put(74,1.5){\makebox(0,0){$\langle$}}
\put(80,1.2){\makebox(0,0){$\bullet$}}
\put(4,6){\makebox(0,0)[b]{$\scriptstyle a$}}
\put(16,6){\makebox(0,0)[b]{$\scriptstyle b$}}
\put(28,6){\makebox(0,0)[b]{$\scriptstyle c$}}
\put(40,6){\makebox(0,0)[b]{$\scriptstyle d$}}
\put(68,6){\makebox(0,0)[b]{$\scriptstyle e$}}
\put(80,6){\makebox(0,0)[b]{$\scriptstyle f$}}
\end{picture})&=&\begin{picture}(72,5)
\put(4,1.5){\line(1,0){30}}
\put(4,1.2){\makebox(0,0){$\bullet$}}
\put(16,1.2){\makebox(0,0){$\bullet$}}
\put(28,1.2){\makebox(0,0){$\bullet$}}
\put(43,1.2){\makebox(0,0){$\cdots$}}
\put(50,1.5){\line(1,0){6}}
\put(56,1.2){\makebox(0,0){$\bullet$}}
\put(56,0.5){\line(1,0){12}}
\put(56,2.5){\line(1,0){12}}
\put(62,1.5){\makebox(0,0){$\langle$}}
\put(68,1.2){\makebox(0,0){$\bullet$}}
\put(4,6){\makebox(0,0)[b]{$\scriptstyle b$}}
\put(16,6){\makebox(0,0)[b]{$\scriptstyle c$}}
\put(28,6){\makebox(0,0)[b]{$\scriptstyle d$}}
\put(56,6){\makebox(0,0)[b]{$\scriptstyle e$}}
\put(68,6){\makebox(0,0)[b]{$\scriptstyle f$}}
\end{picture}\\[4pt]
H^1({\mathfrak{h}},\begin{picture}(84,5)
\put(4,1.5){\line(1,0){42}}
\put(4,1.2){\makebox(0,0){$\bullet$}}
\put(16,1.2){\makebox(0,0){$\bullet$}}
\put(28,1.2){\makebox(0,0){$\bullet$}}
\put(40,1.2){\makebox(0,0){$\bullet$}}
\put(55,1.2){\makebox(0,0){$\cdots$}}
\put(62,1.5){\line(1,0){6}}
\put(68,1.2){\makebox(0,0){$\bullet$}}
\put(68,0.5){\line(1,0){12}}
\put(68,2.5){\line(1,0){12}}
\put(74,1.5){\makebox(0,0){$\langle$}}
\put(80,1.2){\makebox(0,0){$\bullet$}}
\put(4,6){\makebox(0,0)[b]{$\scriptstyle a$}}
\put(16,6){\makebox(0,0)[b]{$\scriptstyle b$}}
\put(28,6){\makebox(0,0)[b]{$\scriptstyle c$}}
\put(40,6){\makebox(0,0)[b]{$\scriptstyle d$}}
\put(68,6){\makebox(0,0)[b]{$\scriptstyle e$}}
\put(80,6){\makebox(0,0)[b]{$\scriptstyle f$}}
\end{picture})&=&\begin{picture}(88,5)(-8,0)
\put(4,1.5){\line(1,0){38}}
\put(4,1.2){\makebox(0,0){$\bullet$}}
\put(24,1.2){\makebox(0,0){$\bullet$}}
\put(36,1.2){\makebox(0,0){$\bullet$}}
\put(51,1.2){\makebox(0,0){$\cdots$}}
\put(58,1.5){\line(1,0){6}}
\put(64,1.2){\makebox(0,0){$\bullet$}}
\put(64,0.5){\line(1,0){12}}
\put(64,2.5){\line(1,0){12}}
\put(70,1.5){\makebox(0,0){$\langle$}}
\put(76,1.2){\makebox(0,0){$\bullet$}}
\put(4,6){\makebox(0,0)[b]{$\scriptstyle a+b+1$}}
\put(24,6){\makebox(0,0)[b]{$\scriptstyle c$}}
\put(36,6){\makebox(0,0)[b]{$\scriptstyle d$}}
\put(64,6){\makebox(0,0)[b]{$\scriptstyle e$}}
\put(76,6){\makebox(0,0)[b]{$\scriptstyle f$}}
\end{picture}\\[4pt]
H^2({\mathfrak{h}},\begin{picture}(84,5)
\put(4,1.5){\line(1,0){42}}
\put(4,1.2){\makebox(0,0){$\bullet$}}
\put(16,1.2){\makebox(0,0){$\bullet$}}
\put(28,1.2){\makebox(0,0){$\bullet$}}
\put(40,1.2){\makebox(0,0){$\bullet$}}
\put(55,1.2){\makebox(0,0){$\cdots$}}
\put(62,1.5){\line(1,0){6}}
\put(68,1.2){\makebox(0,0){$\bullet$}}
\put(68,0.5){\line(1,0){12}}
\put(68,2.5){\line(1,0){12}}
\put(74,1.5){\makebox(0,0){$\langle$}}
\put(80,1.2){\makebox(0,0){$\bullet$}}
\put(4,6){\makebox(0,0)[b]{$\scriptstyle a$}}
\put(16,6){\makebox(0,0)[b]{$\scriptstyle b$}}
\put(28,6){\makebox(0,0)[b]{$\scriptstyle c$}}
\put(40,6){\makebox(0,0)[b]{$\scriptstyle d$}}
\put(68,6){\makebox(0,0)[b]{$\scriptstyle e$}}
\put(80,6){\makebox(0,0)[b]{$\scriptstyle f$}}
\end{picture})&=&\begin{picture}(88,5)(0,0)
\put(4,1.5){\line(1,0){46}}
\put(4,1.2){\makebox(0,0){$\bullet$}}
\put(24,1.2){\makebox(0,0){$\bullet$}}
\put(44,1.2){\makebox(0,0){$\bullet$}}
\put(59,1.2){\makebox(0,0){$\cdots$}}
\put(66,1.5){\line(1,0){6}}
\put(72,1.2){\makebox(0,0){$\bullet$}}
\put(72,0.5){\line(1,0){12}}
\put(72,2.5){\line(1,0){12}}
\put(78,1.5){\makebox(0,0){$\langle$}}
\put(84,1.2){\makebox(0,0){$\bullet$}}
\put(4,6){\makebox(0,0)[b]{$\scriptstyle a$}}
\put(24,6){\makebox(0,0)[b]{$\scriptstyle b+c+1$}}
\put(44,6){\makebox(0,0)[b]{$\scriptstyle d$}}
\put(72,6){\makebox(0,0)[b]{$\scriptstyle e$}}
\put(84,6){\makebox(0,0)[b]{$\scriptstyle f$}}
\end{picture}\\[4pt]
H^3({\mathfrak{h}},\begin{picture}(84,5)
\put(4,1.5){\line(1,0){42}}
\put(4,1.2){\makebox(0,0){$\bullet$}}
\put(16,1.2){\makebox(0,0){$\bullet$}}
\put(28,1.2){\makebox(0,0){$\bullet$}}
\put(40,1.2){\makebox(0,0){$\bullet$}}
\put(55,1.2){\makebox(0,0){$\cdots$}}
\put(62,1.5){\line(1,0){6}}
\put(68,1.2){\makebox(0,0){$\bullet$}}
\put(68,0.5){\line(1,0){12}}
\put(68,2.5){\line(1,0){12}}
\put(74,1.5){\makebox(0,0){$\langle$}}
\put(80,1.2){\makebox(0,0){$\bullet$}}
\put(4,6){\makebox(0,0)[b]{$\scriptstyle a$}}
\put(16,6){\makebox(0,0)[b]{$\scriptstyle b$}}
\put(28,6){\makebox(0,0)[b]{$\scriptstyle c$}}
\put(40,6){\makebox(0,0)[b]{$\scriptstyle d$}}
\put(68,6){\makebox(0,0)[b]{$\scriptstyle e$}}
\put(80,6){\makebox(0,0)[b]{$\scriptstyle f$}}
\end{picture})&=&\begin{picture}(88,5)(0,0)
\put(4,1.5){\line(1,0){46}}
\put(4,1.2){\makebox(0,0){$\bullet$}}
\put(16,1.2){\makebox(0,0){$\bullet$}}
\put(36,1.2){\makebox(0,0){$\bullet$}}
\put(59,1.2){\makebox(0,0){$\cdots$}}
\put(66,1.5){\line(1,0){6}}
\put(72,1.2){\makebox(0,0){$\bullet$}}
\put(72,0.5){\line(1,0){12}}
\put(72,2.5){\line(1,0){12}}
\put(78,1.5){\makebox(0,0){$\langle$}}
\put(84,1.2){\makebox(0,0){$\bullet$}}
\put(4,6){\makebox(0,0)[b]{$\scriptstyle a$}}
\put(16,6){\makebox(0,0)[b]{$\scriptstyle b$}}
\put(36,6){\makebox(0,0)[b]{$\scriptstyle c+d+1$}}
\put(72,6){\makebox(0,0)[b]{$\scriptstyle e$}}
\put(84,6){\makebox(0,0)[b]{$\scriptstyle f$}}
\end{picture}\\
\vdots\\
H^n({\mathfrak{h}},\begin{picture}(84,5)
\put(4,1.5){\line(1,0){42}}
\put(4,1.2){\makebox(0,0){$\bullet$}}
\put(16,1.2){\makebox(0,0){$\bullet$}}
\put(28,1.2){\makebox(0,0){$\bullet$}}
\put(40,1.2){\makebox(0,0){$\bullet$}}
\put(55,1.2){\makebox(0,0){$\cdots$}}
\put(62,1.5){\line(1,0){6}}
\put(68,1.2){\makebox(0,0){$\bullet$}}
\put(68,0.5){\line(1,0){12}}
\put(68,2.5){\line(1,0){12}}
\put(74,1.5){\makebox(0,0){$\langle$}}
\put(80,1.2){\makebox(0,0){$\bullet$}}
\put(4,6){\makebox(0,0)[b]{$\scriptstyle a$}}
\put(16,6){\makebox(0,0)[b]{$\scriptstyle b$}}
\put(28,6){\makebox(0,0)[b]{$\scriptstyle c$}}
\put(40,6){\makebox(0,0)[b]{$\scriptstyle d$}}
\put(68,6){\makebox(0,0)[b]{$\scriptstyle e$}}
\put(80,6){\makebox(0,0)[b]{$\scriptstyle f$}}
\end{picture})&=&
\begin{picture}(84,5)
\put(4,1.5){\line(1,0){30}}
\put(4,1.2){\makebox(0,0){$\bullet$}}
\put(16,1.2){\makebox(0,0){$\bullet$}}
\put(28,1.2){\makebox(0,0){$\bullet$}}
\put(43,1.2){\makebox(0,0){$\cdots$}}
\put(50,1.5){\line(1,0){6}}
\put(56,1.2){\makebox(0,0){$\bullet$}}
\put(56,0.5){\line(1,0){22}}
\put(56,2.5){\line(1,0){22}}
\put(67,1.5){\makebox(0,0){$\langle$}}
\put(78,1.2){\makebox(0,0){$\bullet$}}
\put(4,6){\makebox(0,0)[b]{$\scriptstyle a$}}
\put(16,6){\makebox(0,0)[b]{$\scriptstyle b$}}
\put(28,6){\makebox(0,0)[b]{$\scriptstyle c$}}
\put(56,6){\makebox(0,0)[b]{$\scriptstyle $}}
\put(78,6){\makebox(0,0)[b]{$\scriptstyle e+f+1$}}
\end{picture}
\end{array}$$
and for $r\geq n+1$, there are isomorphisms
$H^r({\mathfrak{h}},{\mathbb{E}})=H^{2n+1-r}({\mathfrak{h}},{\mathbb{E}})$.
Using the same notation for the bundles $H^r({\mathfrak{h}},E)$, the complexes 
of Theorem~\ref{towardsBGG} become
$$\begin{array}{rl}\begin{picture}(72,5)
\put(4,1.5){\line(1,0){30}}
\put(4,1.2){\makebox(0,0){$\bullet$}}
\put(16,1.2){\makebox(0,0){$\bullet$}}
\put(28,1.2){\makebox(0,0){$\bullet$}}
\put(43,1.2){\makebox(0,0){$\cdots$}}
\put(50,1.5){\line(1,0){6}}
\put(56,1.2){\makebox(0,0){$\bullet$}}
\put(56,0.5){\line(1,0){12}}
\put(56,2.5){\line(1,0){12}}
\put(62,1.5){\makebox(0,0){$\langle$}}
\put(68,1.2){\makebox(0,0){$\bullet$}}
\put(4,6){\makebox(0,0)[b]{$\scriptstyle b$}}
\put(16,6){\makebox(0,0)[b]{$\scriptstyle c$}}
\put(28,6){\makebox(0,0)[b]{$\scriptstyle d$}}
\put(56,6){\makebox(0,0)[b]{$\scriptstyle e$}}
\put(68,6){\makebox(0,0)[b]{$\scriptstyle f$}}
\end{picture}
&\xrightarrow{\,\nabla^{a+1}\,}
\begin{picture}(88,5)(-8,0)
\put(4,1.5){\line(1,0){38}}
\put(4,1.2){\makebox(0,0){$\bullet$}}
\put(24,1.2){\makebox(0,0){$\bullet$}}
\put(36,1.2){\makebox(0,0){$\bullet$}}
\put(51,1.2){\makebox(0,0){$\cdots$}}
\put(58,1.5){\line(1,0){6}}
\put(64,1.2){\makebox(0,0){$\bullet$}}
\put(64,0.5){\line(1,0){12}}
\put(64,2.5){\line(1,0){12}}
\put(70,1.5){\makebox(0,0){$\langle$}}
\put(76,1.2){\makebox(0,0){$\bullet$}}
\put(4,6){\makebox(0,0)[b]{$\scriptstyle a+b+1$}}
\put(24,6){\makebox(0,0)[b]{$\scriptstyle c$}}
\put(36,6){\makebox(0,0)[b]{$\scriptstyle d$}}
\put(64,6){\makebox(0,0)[b]{$\scriptstyle e$}}
\put(76,6){\makebox(0,0)[b]{$\scriptstyle f$}}
\end{picture}\\
&\enskip{}\xrightarrow{\,\nabla^{b+1}\,}
\begin{picture}(88,5)(0,0)
\put(4,1.5){\line(1,0){46}}
\put(4,1.2){\makebox(0,0){$\bullet$}}
\put(24,1.2){\makebox(0,0){$\bullet$}}
\put(44,1.2){\makebox(0,0){$\bullet$}}
\put(59,1.2){\makebox(0,0){$\cdots$}}
\put(66,1.5){\line(1,0){6}}
\put(72,1.2){\makebox(0,0){$\bullet$}}
\put(72,0.5){\line(1,0){12}}
\put(72,2.5){\line(1,0){12}}
\put(78,1.5){\makebox(0,0){$\langle$}}
\put(84,1.2){\makebox(0,0){$\bullet$}}
\put(4,6){\makebox(0,0)[b]{$\scriptstyle a$}}
\put(24,6){\makebox(0,0)[b]{$\scriptstyle b+c+1$}}
\put(44,6){\makebox(0,0)[b]{$\scriptstyle d$}}
\put(72,6){\makebox(0,0)[b]{$\scriptstyle e$}}
\put(84,6){\makebox(0,0)[b]{$\scriptstyle f$}}
\end{picture}\\
&\quad{}\xrightarrow{\,\nabla^{c+1}\,}
\begin{picture}(88,5)(0,0)
\put(4,1.5){\line(1,0){46}}
\put(4,1.2){\makebox(0,0){$\bullet$}}
\put(16,1.2){\makebox(0,0){$\bullet$}}
\put(36,1.2){\makebox(0,0){$\bullet$}}
\put(59,1.2){\makebox(0,0){$\cdots$}}
\put(66,1.5){\line(1,0){6}}
\put(72,1.2){\makebox(0,0){$\bullet$}}
\put(72,0.5){\line(1,0){12}}
\put(72,2.5){\line(1,0){12}}
\put(78,1.5){\makebox(0,0){$\langle$}}
\put(84,1.2){\makebox(0,0){$\bullet$}}
\put(4,6){\makebox(0,0)[b]{$\scriptstyle a$}}
\put(16,6){\makebox(0,0)[b]{$\scriptstyle b$}}
\put(36,6){\makebox(0,0)[b]{$\scriptstyle c+d+1$}}
\put(72,6){\makebox(0,0)[b]{$\scriptstyle e$}}
\put(84,6){\makebox(0,0)[b]{$\scriptstyle f$}}
\end{picture}\\[-4pt]
&\qquad\vdots\\
&\qquad{}\xrightarrow{\,\nabla^{e+1}\,}
\begin{picture}(84,5)
\put(4,1.5){\line(1,0){30}}
\put(4,1.2){\makebox(0,0){$\bullet$}}
\put(16,1.2){\makebox(0,0){$\bullet$}}
\put(28,1.2){\makebox(0,0){$\bullet$}}
\put(43,1.2){\makebox(0,0){$\cdots$}}
\put(50,1.5){\line(1,0){6}}
\put(56,1.2){\makebox(0,0){$\bullet$}}
\put(56,0.5){\line(1,0){22}}
\put(56,2.5){\line(1,0){22}}
\put(67,1.5){\makebox(0,0){$\langle$}}
\put(78,1.2){\makebox(0,0){$\bullet$}}
\put(4,6){\makebox(0,0)[b]{$\scriptstyle a$}}
\put(16,6){\makebox(0,0)[b]{$\scriptstyle b$}}
\put(28,6){\makebox(0,0)[b]{$\scriptstyle c$}}
\put(56,6){\makebox(0,0)[b]{$\scriptstyle $}}
\put(78,6){\makebox(0,0)[b]{$\scriptstyle e+f+1$}}
\end{picture}\\
&\enskip\qquad{}\xrightarrow{\,\nabla^{2f+2}\,}
\begin{picture}(84,5)
\put(4,1.5){\line(1,0){30}}
\put(4,1.2){\makebox(0,0){$\bullet$}}
\put(16,1.2){\makebox(0,0){$\bullet$}}
\put(28,1.2){\makebox(0,0){$\bullet$}}
\put(43,1.2){\makebox(0,0){$\cdots$}}
\put(50,1.5){\line(1,0){6}}
\put(56,1.2){\makebox(0,0){$\bullet$}}
\put(56,0.5){\line(1,0){22}}
\put(56,2.5){\line(1,0){22}}
\put(67,1.5){\makebox(0,0){$\langle$}}
\put(78,1.2){\makebox(0,0){$\bullet$}}
\put(4,6){\makebox(0,0)[b]{$\scriptstyle a$}}
\put(16,6){\makebox(0,0)[b]{$\scriptstyle b$}}
\put(28,6){\makebox(0,0)[b]{$\scriptstyle c$}}
\put(56,6){\makebox(0,0)[b]{$\scriptstyle $}}
\put(78,6){\makebox(0,0)[b]{$\scriptstyle e+f+1$}}
\end{picture}\\
&\qquad{}\xrightarrow{\,\nabla^{e+1}\,}\cdots\\
&\qquad\vdots\\
&\xrightarrow{\,\nabla^{a+1}\,}
\begin{picture}(72,5)
\put(4,1.5){\line(1,0){30}}
\put(4,1.2){\makebox(0,0){$\bullet$}}
\put(16,1.2){\makebox(0,0){$\bullet$}}
\put(28,1.2){\makebox(0,0){$\bullet$}}
\put(43,1.2){\makebox(0,0){$\cdots$}}
\put(50,1.5){\line(1,0){6}}
\put(56,1.2){\makebox(0,0){$\bullet$}}
\put(56,0.5){\line(1,0){12}}
\put(56,2.5){\line(1,0){12}}
\put(62,1.5){\makebox(0,0){$\langle$}}
\put(68,1.2){\makebox(0,0){$\bullet$}}
\put(4,6){\makebox(0,0)[b]{$\scriptstyle b$}}
\put(16,6){\makebox(0,0)[b]{$\scriptstyle c$}}
\put(28,6){\makebox(0,0)[b]{$\scriptstyle d$}}
\put(56,6){\makebox(0,0)[b]{$\scriptstyle e$}}
\put(68,6){\makebox(0,0)[b]{$\scriptstyle f$}}
\end{picture},
\end{array}$$
for arbitrary non-negative integers $a,b,c,d,\cdots,e,f$. When all these
integers are zero, this is the Rumin--Seshadri complex. Just the first three
terms in this complex, in the special case when only $a$ is non-zero, are
already essential in~\cite{EG}. For example, if $a=1$, then the first two
differential operators are
$$\sigma\mapsto\nabla_a\nabla_b\sigma+\Phi_{ab}\sigma
\quad\mbox{and}\quad
\phi_{bc}\mapsto
\big(\nabla_a\phi_{bc}-\nabla_b\phi_{ac})_\perp$$
where $\phi_{bc}$ is symmetric and $(\enskip)_\perp$ means to take the 
trace-free part with respect to $J_{ab}$. 
{From} the curvature decomposition and Bianchi identity we find that their
composition is
$$\sigma\longmapsto
V_{ab}{}^d{}_c\nabla_d\sigma+Y_{abc}\sigma,$$
which vanishes in case $V_{ab}{}^c{}_d=0$. In case $\Theta$ is invertible, 
as for the Fubini--Study connection, we conclude that this sequence of 
differential operators is locally exact.

\end{document}